\documentclass[12pt,a4paper]{amsart}

\usepackage{amsmath,amsthm,amsfonts,amssymb,mathrsfs}
\usepackage{subfigure} 
\usepackage{graphicx}
\usepackage{color}

\usepackage{array}
\usepackage{booktabs}
\usepackage{caption}
\usepackage{float}

\newcommand{\ud}{\mathrm{d}}

\usepackage{amsmath}
\usepackage{amssymb}
\usepackage{epsfig}
\usepackage{float}
\usepackage{latexsym,graphics}
\usepackage{dsfont}
\newtheorem{theorem}{Theorem}[section]
\newtheorem{proposition}[theorem]{Proposition}

\newtheorem{definition}[theorem]{Definition}

\newtheorem{remark}[theorem]{Remark}

\newcommand{\Cu}[1]{\mathbf{C^{#1}}}
	
\newcommand{\Cl}[1]{\mathbf{C_{#1}}}			
			
\newcommand{\BL}{\mathbf{BL}}	
\newcommand{\BLN}{\mathbf{BL}}	
\newcommand{\TV}{\mathbf{TV}}	
\newcommand{\Lip}{\mathbf{Lip}} 				
\newcommand{\reali}{\mathbb{R}}
\newcommand{\naturali}{\mathbb{N}}			

\newcommand{\N}{\mathbb{R^{+}}}

\def\sqr#1#2{{\vcenter{\vbox{\hrule height.#2pt\hbox{\vrule width.#2pt height#1pt \kern#1pt\vrule width.#2pt}\hrule height.#2pt}}}}

\renewcommand{\d}[1]{\mathrm{d}{#1}}	
\DeclareMathOperator*{\esssup}{ess\,sup}
	
\newcommand{\norma}[1]{{\left \|#1\right \|}} 		
\renewcommand{\phi}{\varphi}		

\newenvironment{proofof}[1]
{\smallskip\noindent{\textbf{Proof~of~#1.}}
\hspace{1pt}}{\hspace{-5pt}{\nobreak\nobreak\hfill\nobreak
$\square$\vspace{2pt}\par}\smallskip\goodbreak}

\allowdisplaybreaks[4]

\newtheorem{assumptions}[theorem]{Assumptions}

\newcommand{\R}{\mathbb{R}}

\begin{document}

\title[Asynchronous exponential growth in measure spaces]{Asynchronous exponential growth for structured population models in measure space}

\author[C. D\"{u}ll]{Christian D\"{u}ll$^1$}
\author[J. Z. Farkas]{J\'{o}zsef Z. Farkas$^2$}
\author[P. Gwiazda]{Piotr Gwiazda$^3$}
\author[A. Marciniak-Czochra]{Anna Marciniak-Czochra$^4$}

\address[$^1$]{Institute for Mathematics, Heidelberg University, 69120 Heidelberg, Germany}
\email{duell@math.uni-heidelberg.de}
\address[$^2$]{Departament de Matem\`{a}tiques, Universitat Aut\`{o}noma de Barcelona, Bellaterra, 08193, Spain}
\email{JozsefZoltan.Farkas@uab.cat}
\address[$^3$]{Institute of Applied Mathematics and Mechanics,  University of Warsaw, Warszawa 02-097, Poland}
\email{pgwiazda@mimuw.edu.pl}
\address[$^4$]{Institute for Mathematics and Interdisciplinary Center for Scientific Computing (IWR), Heidelberg University, 69120 Heidelberg, Germany}
\email{anna.marciniak@iwr.uni-heidelberg.de}

\date{}
\date{}

\begin{abstract}
This paper studies the asymptotic behaviour of a structured population model on the space of nonnegative Radon measures. Such formulations naturally arise when solutions develop concentration phenomena or when the population is represented by discrete cohorts. Asynchronous exponential convergence of measure solutions towards a one-dimensional global attractor is established. While such results are classical in the $L^1$ setting, their extension to measure spaces requires different compactness and spectral arguments. We identify conditions under which the classical asymptotic behaviour persists in the space of Radon measures endowed with the flat metric, thereby extending the theory of asynchronous exponential growth beyond the classical $L^1$ framework.\\


\end{abstract}


\maketitle

\section{Introduction}
\noindent Mathematical models describing the dynamics of physiologically structured populations  have proved to be useful tools in a range of applications in cell biology, evolution and ecology, see for example \cite{MagalRuan,MetzDiekmann, Thieme, Webb}. The models take the form of partial or integro-differential equations describing the  evolution of the density of individuals distributed with respect to a specific structural variable. The choice of the structuring variable is motivated by applications and reflects the functionally relevant heterogeneity of the study population, such as age, size, phenotypic trait or cell maturity, see e.g. \cite{AcklehFarkas, CalsinaCuadrado, CalsinaCuadrado2, Desvillettes, Perthame}.\\
Classical results for structured population models were obtained in the space of Lebesgue integrable functions (densities), as presented in the seminal monograph \cite{Webb}. The choice of the state space $L^1$ is motivated by applications, as convergence with respect to the natural norm implies convergence of the total population size. Moreover, the spectral theory of positive operators on abstract Lebesgue spaces is well-developed and provides a convenient framework for analysing structured population models, see e.g. \cite{Arendt,Clement,EngelNagel,Sch}. In particular, combining spectral methods with semigroup theory has led to a rich theory describing the long-term behaviour of solutions. Positive irreducible semigroups with suitable compactness properties were shown to possess favourable spectral properties directly linked to asynchronous exponential growth and convergence to finite-dimensional attractors; see e.g. \cite{Arendt,Clement,EngelNagel,Thieme1998JMAA,Thieme1998DCDS,Webb,Webb1987,WebbGrabosch}.
\\
More recently, the study of structured population models has been extended to spaces of nonnegative Radon measures; see, for instance, \cite{BGMC,CCGU,Duell_book,EversHilleMuntean,GLMC,GMC,Piccoli,two_sex}. This more general approach is particularly useful for models for which local or global well-posedness may fail in Lebesgue spaces. One example is a class of quasilinear equations exhibiting a concentration phenomenon. Due to the nonlinear transport term, solutions with $L^1$ initial data may concentrate in finite time, see for example \cite{Ackleh1}. Another example is provided by selection or selection-mutation models asymptotically showing a concentration of mass due to specific non-local feedbacks \cite{BCMC,BGMC,MagalWebb}. At the same time, the idea of representing a heterogeneous population as a sum of masses concentrated in different points of the individual state space can also be motivated from an application perspective. For example, considering a homogeneous initial population naturally leads to defining the initial conditions of the model in terms of a Dirac measure, as introduced already in \cite{MetzDiekmann}.\\
A mathematical framework for the analysis of structured population models formulated on the space of nonnegative Radon measures has been proposed in \cite{GLMC}, using the flat metric (dual bounded Lipschitz distance). The underlying functional analytic theory and its application to showing well-posedness of the models, in particular existence, uniqueness and Lipschitz continuous dependence of solutions on the model ingredients, are presented  in a recent monograph \cite{Duell_book}. As a first consequence, the framework allowed  establishing the stability of numerical schemes based on a particle method, for example the escalator boxcar train (EBT) algorithm, see e.g., \cite{EBT,CGU,GJMCU}.\\
Beyond resolving regularity issues, the measure framework allows for both discrete and continuous distributions, as well as mixtures thereof. Moreover, it naturally extends to general complete separable metric spaces, including graphs \cite{M3AS_measures}. This provides a link to concepts and methods developed in stochastic modelling. In particular, asymptotic properties of conservative measure solutions can be studied using concentrating Feller operators \cite{MR2287895,MR2271485,MR2995659,Thieme_2022}, while transport-type equations are closely related to ergodic properties of Markov processes \cite{MR2857021}. Such connections have recently been used to establish nonexpansiveness of solution semigroups in metrics related to the flat distance \cite{FournierPerthame2,FournierPerthame1}.\\
In this paper we study the asymptotic behaviour of non-conservative structured population models formulated in spaces of measures. While asymptotic properties of conservative measure-valued systems can often be analysed using methods originating from stochastic processes, non-conservative models are more naturally studied using semigroup-theoretic techniques. Such methods have previously been applied to integro-differential equations formulated in measures under the total variation norm \cite{BuergerBomze,MischlerScher}. However, this approach cannot be extended directly to structured population models with transport, since the transport semigroup fails to be strongly continuous with respect to the total variation norm.\\
Our goal is therefore to transfer the classical asymptotic results available in the $L^1$ setting to a measure framework endowed with the flat metric. 
The latter provides a topology in which transport semigroups are strongly continuous and allows us to construct a strongly continuous positive semigroup on an appropriate Banach lattice. Using a dual formulation together with spectral methods for positive semigroups, we establish conditions under which asynchronous exponential growth persists for measure solutions.\\
The assumptions imposed on the model ingredients are comparable to those used in the classical $L^1$ theory \cite{KangShigui,Webb85,Webb1987}. The main difference lies in the use of quasi-compactness rather than eventual compactness, which allows us to treat unbounded state spaces. In particular, unlike several earlier studies \cite{AcklehFarkas,CDF,FarkasHinow}, we do not assume an upper bound on the structuring variable. While this setting is often more natural in applications, it introduces additional challenges in the spectral analysis of the governing semigroup, as already observed for related models formulated on Lebesgue spaces \cite{FarkasHagen}.\\
More precisely, we formulate conditions ensuring that the asynchronous exponential convergence known from the classical $L^1$ framework remains valid in the measure setting equipped with the flat metric. To the best of our knowledge, this is the first result establishing asynchronous exponential growth for a broad class of physiologically structured population models formulated on spaces of Radon measures.\\
In this paper, we focus on a generic structured population model arising in various contexts of mathematical biology. However, the proposed approach is not restricted to this setting and can also be applied to coagulation--fragmentation models formulated in measure spaces, see \cite{MR4330732}, since, after a suitable reformulation of the fragmentation kernel, these models fit naturally into the measure PDE framework considered here \cite[Section~5.3]{Duell_thesis}.
\\

\noindent The paper is structured as follows: First, we introduce the most important measure theoretic concepts, the model formulation in the space of nonnegative 
Radon measures as well as well-posedness results. In Section \ref{section:main_result} we  construct the semigroups on an appropriate Banach lattice and formulate results on their strong continuity and asymptotic behaviour. The latter is proven in Section \ref{section:asymptotic_proof} by establishing quasi-compactness of the semigroup. The paper concludes with an irreducibility criterion ensuring asynchronous exponential growth towards a one-dimensional attractor.

\section{Problem Formulation in the setting of measures}
\label{section:problem}

\noindent We consider a linear structured population model formulated in the space of finite nonnegative measures $\mathcal{M}^+(\R^+)$. Specifically, we will work with the following model 
\begin{align}\label{Model}
\left\{\begin{array}{lll}
\partial_t \mu_t+\partial_x \left(b(x) \, \mu_t \right)+c(x)\,\mu_t &= \displaystyle\int_{\R^+} (\eta(x))(y)\, \ud\mu_t(y),& (t,x) \in  \R^+ \times \R^+,\\ 
    b(0)D_{\lambda}\mu_t(0)&= 0& t \in \R^+,\\
    \mu_{t=0}&=\mu_0,
\end{array}\right.
  \end{align}
where $D_{\lambda}\mu_t$ denotes the Radon-Nikodym derivative of $\mu_t$ with respect to the Lebesgue measure $\lambda$ at the point $x=0$ and $\mu_0 \in  \mathcal{M}^+(\R^+)$ is some initial measure. The model involves a transport operator describing the development of individuals with respect to a physiological structuring variable, an integral operator characterising the birth/recruitment process, and a li\-ne\-ar decay term accounting for individual mor\-ta\-li\-ty. 
We equip the state space $\mathcal{M}^+(\R^+)$ with the so called flat norm
    \begin{align}
    \label{new_flat}
    ||\mu||_{\BLN^*}:=\sup\left\{\int_{\R^+}\phi\,\ud \mu:\,||\phi||_{\BLN}\le 1\right\},
    \end{align}
where  $\|\cdot\|_{\BLN}$ is a norm on the space $\BL(\mathbb{R}^+)$ of bounded Lipschitz functions, given by 
\begin{align}
    \label{ourBL}
    \|\phi\|_{\BL}:=\max\{ \|\phi\|_{\infty}, |\phi|_{\Lip}\}.
    \end{align}   
Here we used the notation 
    \begin{align*}
    \|\phi\|_{\infty}:=\displaystyle\sup_{x\in\R^+}\left|\phi(x)\right|,\qquad \text{and} \qquad
    |\phi|_{\Lip}:=\displaystyle\sup_{x\neq y}\frac{|\phi(x)-\phi(y)|}{|x-y|}.
    \end{align*}
The flat norm provides a convenient setting for considering differential equations in a measure setting, see for example \cite{MR4330732,Duell_book, M3AS_measures}.

\begin{remark}
\begin{enumerate}
\item [i)]  The cone $\mathcal{M}^+(\R^+)$ defines a natural ordering "$\leq$" on the space of finite signed measures $\mathcal{M}(\R^+)$ via
\begin{align}
\label{partial_order}
       \mu\leq \nu \hspace{0.3cm} \Longleftrightarrow \hspace{0.3cm} \nu-\mu\in \mathcal{M}^+(\R^+), \hspace{0.3cm} i.e.\hspace{0.3cm} \mu(A)\leq \nu(A)\hspace{0.3cm} \forall A\in \mathcal{B}(\R^+),
    \end{align}
with $\mathcal{B}(\R^+)$ denoting the usual Borel $\sigma$-algebra. 
\item [ii)] \label{remark:measure_space} The partial order \eqref{partial_order} can be extended to the closure 
    \begin{align*}
    E:=\overline{\mathcal{M}(\R^+)}^{\|\cdot\|_{\BLN^*}}
    \end{align*}
with positive cone $E^+=\mathcal{M}^+(\R^+)$ as $\mathcal{M}^+(\R^+)$ is closed with respect to $\|\cdot\|_{\BLN^*}$, see \cite{vanGaans} and \cite[Thm G.42]{Duell_book}. Consequently, the space $\mathcal{M}^+(\R^+)$ can be seen as the positive cone of the Banach lattice $(E,\|\cdot\|_{\BLN^*})$. Furthermore, in view of  \cite[Thm 1.41]{Duell_book} we  have $E^*=BL(\R^+)$.
\end{enumerate}
\end{remark}

\noindent A suitable notion of continuity for measures is given by the concept of narrow continuity, which requires convergence in duality with all bounded continuous functions, i.e. a sequence $(\mu^n)_{n\in\naturali} \subset \mathcal M^+(\reali^+)$ \textbf{converges narrowly} to a measure $\mu \in \mathcal M^+(\reali^+)$, if and only if 
\begin{equation*}
\lim_{n \to \infty} \int_{\reali^+} \phi(x) \,\ud\left(\mu^n - \mu\right) (x) = 0, \quad \forall\phi \in \Cl{b}(\reali^+).
\end{equation*}
Similarly, we say that a mapping $\mu_{\bullet} : [0,T] \mapsto {\mathcal M}^+(\reali^+)$ is \textbf{narrowly continuous}, if for every $\phi \in \Cl{b}(\reali^+)$ the function
    \begin{align*}
    f:[0,T]\mapsto \reali,\quad f(t) = \int_{\reali^+} \phi(x) \,\ud \mu_t(x)
    \end{align*}
is continuous. We remark that according to Theorem 1.57 in \cite{Duell_book} narrow continuity on the cone $\mathcal{M}^+(\R^+)$ is equivalent to convergence with respect to the flat norm $\|\cdot\|_{BL^*}$.\\
\noindent Solutions to model \eqref{Model} are to be understood in a weak sense. In particular, for a finite time interval $[0,T]$ we introduce the following 
\begin{definition}
 \label{def:WeakSolution}
Given $T>0$, a function $\mu_{\bullet} \colon [0,T] \to \mathcal{M}^+(\R^+)$ is a \textbf{measure solution}
of model \eqref{Model}, if $\mu_{\bullet}$ is narrowly continuous, and for all $\phi \in (\Cu{1}
 \cap \BL) \left(\reali^+\times \reali^+\right)$ the following equality holds 
 \begin{align}
 \begin{split}
 \label{Formulation:WeakSolution}
  & \int_{\reali^+} \phi(T,x) \; \d{\mu_T}(x) - \int_{\reali^+}
  \phi(0,x) \; \d{\mu_0}(x)\\
  =&\int_0^T \int_{\reali^+} \left(
   \partial_t \phi(t,x) + b(x) \; \partial_x
   \phi(t,x) - c(x) \; \phi(t,x) \right)\,\ud\mu_t(x)\,\ud{t}   \\
  &+ \int_0^T \int_{\reali^+} \left( \int_{\reali^+} \phi(t,y)
   \d{\left[\eta(x)\right]} (y) \right)\, \ud\mu_t(x)\,\ud{t}.
   \end{split}
\end{align}
\end{definition}
\noindent
 
\begin{remark}
\begin{itemize}
    \item [i)] The term $\int_{\reali^+} \phi(t,y) \,\ud{\left[\eta(x)\right]}(y)$ denotes the integral of $\phi(t,y)$ with respect to the measure $\eta(x)$ in the variable $y$. The other integrals are to be understood similarly.
    \item [ii)] A measure solution in the sense of Definition \eqref{def:WeakSolution} also satisfies the weak formulation \eqref{Formulation:WeakSolution} with the final time $T$ replaced by $s\in [0,T]$. This can be shown by choosing a suitable cut-off function $h_{\varepsilon}:[0,T]\to [0,1]$ which satisfies $h_{\varepsilon}\!\mid_{[0,s]}\equiv 1$ and  decreases linearly to $0$ in $[s,s+\varepsilon]$.
 \end{itemize}
 \end{remark}

\noindent We impose the following assumptions on the model parameters.
\begin{assumptions}\label{Assum} 
\ \begin{itemize}
\item[(i)] $c \in  \BL(\R^+)\cap C^1(\R^+)$.
\item[(ii)] $x\mapsto \eta(x) \in  \BL\left(\R^+; (\mathcal{M}^+(\R^+), \norma{\cdot}_{\BLN^*})\right)$.
\item[(iii)] $b \in  \BL(\R^+)\cap C^1(\R^+)$, $b>0$. 
\item[(iv)] $b' \leq 0$ .
\item [(v)] There exists a constant $\kappa>0$, such that 
    \begin{align*}
        c(x) \geq \left|c'(x)\right| + \kappa\qquad \forall x \in \R^+.
    \end{align*}
\end{itemize}
\end{assumptions}
\noindent 
We remark that a norm in the space $\BL\left(\R^+; (\mathcal{M}^+(\R^+), \norma{\cdot}_{\BLN^*})\right)$ is defined as
\begin{displaymath}
\norma{\eta}_{\BLN} = \sup_{x\in\reali^+}\norma{\eta(x)}_{\BLN^*} + \Lip(\eta),
\end{displaymath}
where $\Lip(\eta)$ denotes the usual Lipschitz constant of $\eta$. \\

\noindent Assumptions (i)–(iii) ensure the well-posedness of models of the form \eqref{Formulation:WeakSolution}. Similar results for a comparable model on $\mathbb{R}^+$ with a slightly less general integral operator can be found in Section 2.2 of \cite{Duell_book}, particularly in Theorem 2.19 and Lemma 2.25. Alternatively, we refer to \cite[Section 3.3]{Duell_thesis}, where a well-posedness theory is developed for a non-autonomous model on $\mathbb{R}^d$ involving the integral operator. The additional assumptions (iv) and (v) are required for the analysis of the asymptotic behaviour of solutions.\\

\noindent
For the reader's convenience, we recall the well-posedness results in a form adapted to our problem.

\begin{theorem} \label{existenceAgnieszka}
Let Assumptions \ref{Assum}(i)--(iii) hold and fix $T>0$. Then the model \eqref{Model} (considered on the time interval $[0,T]$) generates a semigroup
\[
\{\mathcal{T}(t)\}_{t\in[0,T]}:\mathcal{M}^+(\mathbb{R}^+)\to\mathcal{M}^+(\mathbb{R}^+)
\]
with the following properties.
\begin{enumerate}
\item $\mathcal{T}(0)=\mathbf{Id}$, and for all $t_1,t_2\in[0,T]$ such that $t_1+t_2\in[0,T]$,
\[
\mathcal{T}(t_1)\circ\mathcal{T}(t_2)=\mathcal{T}(t_1+t_2).
\]

\item For every $t\in[0,T]$ and every $\mu_0\in\mathcal{M}^+(\mathbb{R}^+)$, let
$\mu_t=\mathcal{T}(t)\mu_0$. Then the trajectory
$t\mapsto\mu_t$ is the unique solution of \eqref{Model} with initial condition $\mu_0$ in the sense of Definition~\ref{def:WeakSolution}. Moreover, $t\mapsto\mu_t$ is Lipschitz continuous and satisfies
\[
\|\mathcal{T}(t)\mu_0-\mu_0\|_{\BLN^*}
\le C_1(t)\,t\,\|\mu_0\|_{\TV},
\]
where $\|\cdot\|_{\TV}$ denotes the total variation norm and $C_1$ depends only on $t$ and on the norms of the model coefficients.

\item For every $t\in[0,T]$ and every $\mu_1,\mu_2\in\mathcal{M}^+(\mathbb{R}^+)$,
\[
\|\mathcal{T}(t)\mu_1-\mathcal{T}(t)\mu_2\|_{\BLN^*}
\le
C_2(t)\,
\|\mu_1-\mu_2\|_{\BLN^*},
\]
where $C_2$ depends only on $t$ and on the norms of the model coefficients.
\end{enumerate}
\end{theorem}

\begin{remark}
The constants $C_1,C_2$ appearing in Proposition \ref{existenceAgnieszka} are continuous and monotonically increasing functions of time.
\end{remark}

\noindent We conclude this section by investigating the regularity properties induced by the recruitment operator. The result below shows that, for every positive time, the solution cannot accumulate mass near the origin faster than linearly with respect to the Lebesgue measure. In particular, the Radon--Nikodym derivative at the origin remains bounded. 

\begin{proposition}
\label{prop:abs_continuity}
Let $\mu_{\bullet}$ be a measure solution to \eqref{Model} in the sense of Definition~\ref{def:WeakSolution}.
Then for every
$t\in(0,T]$ there exist $\varepsilon(t)>0$ and $C(t)>0$ such that $\mu_t([0,\varepsilon])\le
C(t)\varepsilon $ 
for all
$0<\varepsilon<\varepsilon(t)$. Consequently,
\[
\limsup_{\varepsilon\to0}
\frac{\mu_t([0,\varepsilon])}{\varepsilon}
<\infty .
\]
\end{proposition}

\begin{proof}
Fix $t>0$. We show that there exist $\varepsilon(t)>0$ and $C(t)>0$ such that
\[
\mu_t([0,\varepsilon])\le C(t)\,\varepsilon
\]
for all $0<\varepsilon<\varepsilon(t)$.

\noindent By Assumptions~\ref{Assum} $b$ is continuous and strictly positive, and thus
for $\delta_1=\frac 1 2 b(0)>0$ there exists $\widetilde\varepsilon_1>0$  with  $b(x)\ge \delta_1$
 for all $x\in[0,\widetilde\varepsilon_1]$. 
Similarly, for $\tilde \varepsilon_2:=\tilde \varepsilon_1+tb(0)$, we can choose $\delta_2>0$ so that 
    \begin{align}
    \label{delta_2_bound}
        b(x)\geq \delta_2\qquad\text{for all }x\in[0,\widetilde\varepsilon_2].
    \end{align}
As $b'\leq 0$ by Assumptions~\ref{Assum}, $\delta_2>\delta_1$. Choose 
\[
\varepsilon(t)
<
\frac12\min\{\widetilde\varepsilon_1,\delta_2 t\}.
\]
\noindent Now let $0<\varepsilon<\varepsilon(t)$ and choose $\psi_\varepsilon\in BL(\R^+)\cap C^1(\R^+)$ satisfying
\begin{align*}
0\le \psi_\varepsilon\le 1,\qquad
\psi_\varepsilon\equiv 1 \ \text{on }[0,\varepsilon],
\qquad
\operatorname{supp}\psi_\varepsilon\subset[0,2\varepsilon].
\end{align*}
Then
\begin{align}
\label{estimate_mu_t}
\mu_t([0,\varepsilon])
\le
\int_{\R^+}\psi_\varepsilon(x)\,\ud\mu_t(x).
\end{align}

\noindent Let $\tilde \varphi_{\varepsilon,t}$ denote the solution of the adjoint
transport-decay problem
\begin{align}
\label{adjusted_dual_problem}
\left\{
\begin{array}{lll}
\partial_{\tau}\tilde \varphi_{\varepsilon,t}
+b\,\partial_x\tilde \varphi_{\varepsilon,t}
-c\,\tilde \varphi_{\varepsilon,t}
=0,
& \text{in } [0,t]\times\R^+,\\[1mm]
\tilde \varphi_{\varepsilon,t}(t,\cdot)
=\psi_\varepsilon,
& \text{in } \R^+.
\end{array}
\right.
\end{align}
It holds that
\[
\tilde \varphi_{\varepsilon,t}\in
C^1([0,t]\times\R^+)\cap BL([0,t]\times\R^+).
\]

\noindent For $\tau\in [0,t]$ the method of characteristics implies
\[
\tilde \varphi_{\varepsilon,t}(\tau,x)
=
\psi_\varepsilon(X_b(t-\tau,x))
\exp\!\left(
-\int_\tau^t
c(X_b(s-\tau,x))\,\ud s
\right),
\]
where $X_b$ denotes the flow of the vector field $b$, i.e. it solves the ODE
	\begin{align}
    \label{flow}
	\partial_s X_b(s,x) = b(X_b(s, x)) \quad \quad X_b(0, x) = x.
	\end{align}
We estimate
\begin{align}
\label{estimate_tilde_varphi}
0\le
\tilde \varphi_{\varepsilon,t}(\tau,x)
\le
e^{t\|c\|_\infty}
\mathbf 1_{\{X_b(t-\tau,x)\in[0,2\varepsilon]\}}.
\end{align}

\noindent
Using \eqref{flow} we compute
\begin{align}
\label{flow_equation}
X_b(t-\tau,x)=x+\int_0^{t-\tau}b(X_b(s,x))\,\ud s.
\end{align}
Now if $x\geq\tilde \varepsilon_1$, then the strict positivity of $b$ together with \eqref{flow_equation} implies 
    \begin{align}
    \label{estimate_above_eps_1}
        X_b(t-\tau,x)\geq x\geq\tilde \varepsilon_1>2\varepsilon\qquad \forall x\in [\tilde\varepsilon_1,\infty).
    \end{align}
To show a similar bound if  $x\in[0,\tilde\varepsilon_1)$, note that  \eqref{flow_equation}  together with $b'\leq 0$ implies 
\begin{align*}
    X_b(s,x)\leq x+sb(0)\leq \tilde \varepsilon_1+tb(0)=\tilde \varepsilon_2\qquad \forall \,0\leq s\leq t.
    \end{align*}
Thus, for all $\tau\leq t-\frac {2\varepsilon}{\delta_2}$ we compute with \eqref{flow_equation} and \eqref{delta_2_bound}
\begin{align}
\label{estimate_below_eps_1}
X_b(t-\tau,x)\ge x+\delta_2(t-\tau)>2\varepsilon\qquad \forall x\in [0,\tilde \varepsilon_1).
\end{align}
Combining  \eqref{estimate_above_eps_1} and \eqref{estimate_below_eps_1} with the estimate \eqref{estimate_tilde_varphi} yields
    \begin{align*} 
\tilde \varphi_{\varepsilon,t}(\tau,\cdot)\equiv0 \qquad \forall \tau\leq t-\frac {2\varepsilon} {\delta_2}.
    \end{align*}
\noindent Using $\tilde \varphi_{\varepsilon,t}$ as a test function in the weak
formulation \eqref{Formulation:WeakSolution} with the upper time boundary replaced by $t$  yields
    \begin{align*}
&\int_{\R^+}\psi_\varepsilon(x)\,\ud\mu_t(x)
=\int_{\R^+}\tilde\varphi_{\varepsilon,t}(t,x)\,\ud\mu_t(x)\\
=&\int_0^t\tilde\varphi_{\varepsilon,t}(0,x)\,\ud\mu_0(x)+
\int_0^t
\int_{\R^+}
\left(
\int_{\R^+}
\tilde \varphi_{\varepsilon,t}(\tau,y)\,
\ud[\eta(x)](y)
\right)
\ud\mu_\tau(x)\,\ud\tau
\end{align*}
as the transport-decay terms cancel
by construction of $\tilde \varphi_{\varepsilon,t}$. Since $\tilde \varphi_{\varepsilon,t}(\tau,\cdot)\equiv0$ for $\tau\leq t-\frac {2\varepsilon}{\delta_2}$ this reduces to
\[
\int_{\R^+}\psi_\varepsilon(x)\,\ud\mu_t(x)
=
\int_{t-\frac{2\varepsilon}{\delta_2}}^t
\int_{\R^+}
\left(
\int_{\R^+}
\tilde \varphi_{\varepsilon,t}(\tau,y)\,
\ud[\eta(x)](y)
\right)
\ud\mu_\tau(x)\,\ud\tau.
\]

\noindent Using the boundedness of $\tilde \varphi_{\varepsilon,t}$ we obtain
\[
\int_{\R^+}\psi_\varepsilon(x)\,\ud\mu_t(x)
\le
e^{t\|c\|_\infty}\|\eta\|_{\BLN}
\int_{t-\frac{2\varepsilon}{\delta_2}}^t
\|\mu_\tau\|_{\TV}\,\ud\tau.
\]
The narrow continuity of $\tau\mapsto\mu_\tau$ on $[0,t]$ gives the bound
\[
R_t:=
\sup_{\tau\in[0,t]}
\|\mu_\tau\|_{\TV}
<\infty.
\]
Together with estimate \eqref{estimate_mu_t} we conclude 
\[
\mu_t([0,\varepsilon])
\le
\int_{\R^+}\psi_\varepsilon(x)\,\ud\mu_t(x)
\le
\frac{2}{\delta_2}
e^{t\|c\|_\infty}
M_\eta R_t\,\varepsilon=:C(t)\varepsilon.
\]

\end{proof}

\begin{remark}
Although the estimate above controls the growth of mass near the origin, it does not imply local absolute continuity of $\mu_t$. In fact, Assumptions~\ref{Assum}~(i)--(iii) allow for purely atomic recruitment kernels. This should be contrasted with the boundary recruitment model considered in \cite[Lemma 2.20]{Duell_book}, where $\eta(x)=a(x)\delta_0$. In that setting, the time of recruitment parametrises the position along characteristics, which yields a genuine regularising effect and local absolute continuity near the origin.

\noindent For the present distributed recruitment model such a conclusion cannot be expected in general. Indeed, kernels of the form
$\eta(x)=\delta_{(x-h)_+}$, 
preserve atomic components of the solution and therefore do not exhibit the same regularising mechanism.
\end{remark}

\begin{remark}
After a careful inspection of the proof of Proposition \ref{prop:abs_continuity} we can relax the assumption on the strict positivity of $b$ as we only need that the support of $b$ should be large enough that $b$ has no zeros in the scope of the model. So if we assume $b(0)>0$ and that the first zero $N\in \R^+\cup\{\infty\}$ of $b$ satisfies $N>Tb(0)$, then we can choose  $\tilde\varepsilon_1\in (0,N)$ so small that 
    \begin{align*}
    \tilde\varepsilon_2:=\tilde\varepsilon_1+Tb(0)<N.
    \end{align*}
After selecting the corresponding $\delta_1$ and $\delta_2$ the rest of the proof follows the same lines. \\
Note that the other results in this paper only require $b(0)>0$, so that this more general assumption does not affect their validity. 
\end{remark}

\section{Main Results}
\label{section:main_result}
\noindent
In this section we formulate the results of this paper concerning  semigroup properties and  asymptotic behaviour of the model solutions.
\noindent First, we show that the semigroup $\mathcal{T}(t)$ can be extended to the whole time interval $[0,\infty)$ due to the arbitrary choice of $T<\infty$.

\begin{proposition}\label{stronglycontinuous}
Under Assumptions \ref{Assum}(i)-(iii) there exists a unique solution\\
$\mu_{\bullet} \colon [0,T] \to
\mathcal{M}^+(\R^+)$ of the weak formulation \eqref{Formulation:WeakSolution}, which coincides with a trajectory of a strongly continuous semigroup $\mathcal{T}(t)$ on the Banach space $E=\overline{\mathcal{M}(\mathbb{R}^+)}^{||\cdot ||_{\BLN^*}}$, defined for all $t\in \mathbb{R}^+$.
\end{proposition}
\begin{proof}
The assertion for the positive cone $\mathcal{M}^+(\R^+)$ follows from Proposition \ref{existenceAgnieszka} and from the observation that, in fact, the semigroup $\mathcal{T}(t)$ can be defined on the whole interval $[0,\infty)$, due to the arbitrary choice of $T<\infty$. More precisely, any two semigroups $\mathcal{T}^{T_1}(t)$ and $\mathcal{T}^{T_2}(t)$ defined for $t\in [0,T_1]$ and $t\in [0,T_2]$, respectively,   coincide on the interval $[0, \min\{T_1,T_2\}]$ by uniqueness as the corresponding solutions both satisfy the weak formulation. Hence, the semigroup $\mathcal{T}(t)$ of solutions can be extended in time to the whole interval $[0,\infty)$ while preserving  the Lipschitz continuity with respect to time and initial data. To account for initial data in the whole state space $\mathcal M(\R^+)$, we apply the Hahn-Jordan decomposition of a measure into its negative and positive part, and use the linearity of the problem.  In a last step, the Lipschitz operator is extended to the closure $\overline{\mathcal{M}(\mathbb{R}^+)}^{||\cdot ||_{\BLN^*}}=E$, see \cite[Thm. 2.6]{Amman}.\\
Next, we prove that $\mathcal{T}(t)$ defines a strongly continuous semigroup on the whole state space $E$. To this end, let $\varepsilon>0 $ and take an approximation of $\mu_0 \in E$ by  measures $\mu_0^{\varepsilon}\in \mathcal{M}(\R^+)$, see e.g. \cite[Theorem 8.1]{M3AS_measures}, such that
    \begin{align*}
 \norma{\mu_0 - \mu_0^{\varepsilon}}_{{\BLN}^*}<\frac{\varepsilon}{2(C_2(1)+1)},
 \end{align*}
with $C_2$ given in  Proposition \ref{existenceAgnieszka}. We compute 
\begin{eqnarray*}
 \norma{\mathcal{T}(t)\,\mu_0 - \mu_0}_{{\BLN}^*} \leq  \norma{\mathcal{T}(t)\,\mu_0 - \mathcal{T}(t)\,\mu_0^{\varepsilon}}_{{\BLN}^*} + \norma{\mathcal{T}(t)\,\mu_0^{\varepsilon} - \mu_0^{\varepsilon}}_{{\BLN}^*}+ \norma{\mu_0 - \mu_0^{\varepsilon}}_{{\BLN}^*}. 
\end{eqnarray*}
Applying the estimates (2) and (3) from Proposition \ref{existenceAgnieszka} and setting 
\begin{align*}
\delta= \min\left\{\frac{\varepsilon}{2\,C_1(1)\,\norma{\mu_0^{\varepsilon}}_{\TV}},1\right\}
\end{align*}
yields for all $t<\delta$ 
    \begin{align*}
    \norma{\mathcal{T}(t)\,\mu_0 - \mu_0}_{{\BLN}^*}< \varepsilon,
    \end{align*}
i.e. strong continuity on $E$.
\end{proof}

\begin{remark}
Although the weak formulation \eqref{Formulation:WeakSolution} is only well-defined for measures $\mu\in\mathcal{M}(\R^+)$, the semigroup $\mathcal{T}(t)$ admits a continuous extension to $E$. This extension is obtained by approximating elements of $E$ by measures and passing to the limit of the corresponding solutions in the $\BLN^*$ topology. For the purpose of this paper, however, it suffices to consider initial data in $\mathcal{M}^+(\R^+)$. Since this cone is closed with respect to the $\BLN^*$ norm (see Remark \ref{remark:measure_space}(ii)), solutions starting in $\mathcal{M}^+(\R^+)$ remain in the cone for all times and therefore satisfy the weak formulation \eqref{Formulation:WeakSolution}.
 \end{remark}

\noindent As we want to take advantage of the semigroup methods, we rewrite model \eqref{Model} to an abstract Cauchy problem on the separable state space $E=\overline{\mathcal{M}(\mathbb{R}^+)}^{||\cdot ||_{\BLN^*}}$:
    \begin{align}\label{Cauchy}
    \left\{
    \begin{array}{lll}
    \frac{\ud \mu_t}{\ud t}&=\left(\mathcal{B}+\mathcal{C}+\mathcal{N}\right)\mu_t,&(t,x)\in\R^+\times\R^+\\
    b(0)D_{\lambda}\mu_t(0)&= 0,&t\in \R^+\\
    \mu_{t=0}&=\mu_0,
    \end{array}
    \right.
    \end{align}
where we define
\begin{align}
\mathcal{B}\,\mu&=-\frac{\partial}{\partial x}\left(b\,\mu\right),  \hspace{15mm}  \label{Adef} \\
\mathcal{C}\,\mu&=-c\,\mu,  \\
\mathcal{N}\,\mu&=\int_{\mathbb{R}^+}(\eta(x))(\cdot)\,\ud \mu(x).\label{Cdef}
\end{align}

\noindent We clarify the operators.
\begin{proposition}
\label{prop:N_well-posed}
Let $\mathcal{N}:\mathcal{M}(\R^+)\to\mathcal{M}(\R^+)$ be the linear operator introduced in \eqref{Cdef}. Then the following statements hold.
    \begin{itemize}
        \item [i)] For $\mu\in \mathcal{M}(\R^+)$ the integral in \eqref{Cdef} is a well-defined Bochner integral on $E$.
        \item[ii)] $\mathcal{N}$ is a positive operator, i.e. it maps $\mathcal{M}^+(\R^+)$ into $\mathcal{M}^+(\R^+)$. Furthermore, $\mathcal{N}\mathcal{M}(\R^+)\subset\mathcal{M}(\R^+)$.
        \item [iii)] We can extend $\mathcal{N}$ to a well-defined operator $E\to E$.
    \end{itemize}
\end{proposition}

\begin{proof}
To show i) consider $\mu\in \mathcal{M}(\R^+)$. The space $E$ is a separable Banach space, so  according to Pettis Theorem \cite[Theorem, Section V, \S 4]{Yosida},  the concept of strong measurability coincides with the easier to handle weak measurability. Consequently, it is sufficient to check that the map $   \mathbb{R}^+\ni \mapsto\eta(x)\in E$ is weakly measurable, i.e. for all $f\in E^*=BL(\R^+)$ (cf. \cite[Thm 1.41]{Duell_book}), 
    \begin{align*}
    \R^+\ni x\mapsto \langle \eta(x),f\rangle_{E,E^*}
    \end{align*}
is measurable. Here, $\langle\cdot,\cdot\rangle_{E,E^*}$ denotes the usual dual pairing. As $\eta$ is assumed to be in $BL(\R^+; \mathcal M^+(\R^+))$, this map is not only measurable but even Lipschitz continuous since for all $x_1, x_2 \in \R^+$
	\begin{align*}
		\big|\big\langle \eta(x_1) - \eta(x_2),f\big\rangle_{E,E^*}\big| \leq& \|f\|_{E^*}  \left\|\eta(x_1)- \eta(x_2)\right\|_{BL(\R^+)^*}
		 \leq&  \|f\|_{E^*}  |\eta|_{\Lip}\, |x_1-x_2|. 
	\end{align*}
In particular, the map $\R^+\ni x\mapsto \eta(x)$ is strongly measurable. Furthermore, 
    \begin{align*}
    \int_{\R^+}\|\eta(y)\|_{\BLN^*}\, \ud \mu(y)\leq \|\eta\|_{\BLN}\|\mu\|_{\BLN^*}<\infty,
    \end{align*}
so that according to \cite[Theorem 1, Section V \S 5]{Yosida} the map $x \mapsto \eta(x)$ is $\mu$-integrable and $\int_{\R^+} (\eta (y))(\cdot)  \ud \mu(y)$ is a well defined Bochner integral on $E$.\\
In order to show $\mathcal{N}\mu\in \mathcal{M}^+(\R^+)$ for $\mu\in \mathcal{M}^+(\R^+)$, we use the characterisation of $\mathcal{M}^+(\R^+)$ given by Theorem 1.71 in \cite{Duell_book}, i.e.
    \begin{align*}
    E\cap BL(\R^+)^*_+=\mathcal{M}^+(\R^+),
    \end{align*}
where 
    \begin{align*}
        BL(\R^+)^*_+:=\{T\in BL(\R^+)^*\mid T(\psi)\geq 0 \text{ for all } \psi\in BL(\R^+), \psi\geq 0\}
    \end{align*}
denotes the space of positive linear functionals on $BL(\R^+)$. As we already know $\mathcal{N}\mu\in E$ by i), it remains to verify  $\mathcal{N}\mu\in BL(\R^+)^*_+$. Let $\psi\in BL(\R^*)$. Bochner integrals commute with bounded operators (cf. \cite[Corollary 2, Section V §5]{Yosida}),  so for all $\nu\in \mathcal{M}^+(\R^+)$
    \begin{align*}
    \langle \mathcal{N}\nu,\psi\rangle_{E,E^*}=\int_{\R^+}\langle \eta(x),\psi\rangle_{E,E^*}\ud\nu(x)\geq 0;
    \end{align*}
i.e.  $\mathcal{N}\mu\in E\cap BL(\R^+)^*_+=\mathcal{M}^+(\R^+)$, proving (ii). The second part of the statement follows directly from the Hahn-Jordan decomposition Theorem.\\
To show iii), let $\mu\in E$  with approximating sequence $(\mu^n)_{n\in \mathbb{N}}\subset \mathcal{M}(\R^+)$. Define
    \begin{align}
    \label{extension_of_C}
\mathcal{N}\mu:=\lim_{n\to\infty}\mathcal{N}\mu^n\in E.
    \end{align}
We show that the limit in \eqref{extension_of_C} exists by proving that the sequence $(\mathcal{N}\mu^n)_{n\in \N}$ is Cauchy. Indeed, let $\varepsilon>0$, $\|\psi\|_{\BLN}\leq  1$ and let $N\in \N$ with $\|\mu^n-\mu^m\|_{\BLN^*}< \varepsilon/\|\eta\|_{\BLN}$ for all $n,m\geq N$. Using again that Bochner integrals commute with linear operators we compute
    \begin{align}
    \label{C_Cauchy}
    \begin{split}
\int_{\R^+}\psi(x)\,\ud\left[\mathcal{N}\mu^n-\mathcal{N}\mu^m\right](x)
=&\int_{\R^+}\psi(x)\,\ud\left[ \int_{\R^+}\eta(y)\,\ud(\mu^n-\mu^m)(y)\right](x)\\
=&\int_{\R^+}\int_{\R^+}\psi(x)\,\ud [\eta(y)](x)\,\ud(\mu^n-\mu^m)(y)\\
\leq&\|\psi\|_{\BLN}\|\eta\|_{\BLN}\|\mu^n-\mu^m\|_{\BLN^*}\\
=&\|\eta\|_{\BLN}\|\mu^n-\mu^m\|_{\BLN^*}<\varepsilon.
\end{split}
    \end{align}
Taking the supremum over all such $\psi$ yields $\|\mathcal{N}\mu^n-\mathcal{N}\mu^m\|_{\BLN^*}<\varepsilon$, i.e. the sequence $(\mathcal{N}\mu^n)_{n\in \N}$ is indeed Cauchy and the limit exists in $E$. 
\end{proof}

\begin{remark}
\label{rem:transport_for_measures}
As measures are not strongly differentiable, operator $\mathcal{B}$ only makes sense in a weak setting. In particular, for $\mu\in \mathcal{M}(\R^+)$ the measure  $\mathcal{B}\mu$ is defined via partial integration, i.e. for $\varphi\in BL(\R^+)\cap C^1(\R^+)\subset E^*$ 
    \begin{align*}
    \langle \mathcal{B} \mu,\varphi\rangle_{E,E^*}= \int_{\R^+}\varphi(x)\ud[\mathcal{B}\mu](x):=\int_{\R^+}b(x)\partial_x\varphi(x)\ud \mu(x).
    \end{align*}
Using an approximation argument as in the proof of Proposition \ref{prop:N_well-posed}, we can extend $\mathcal{B}$ to an operator $E\to E$.
\end{remark}

\noindent
\begin{remark}
\begin{itemize}
    \item [i)]
According to Proposition \ref{stronglycontinuous} the semigroup of solutions $\mathcal{T}(t):E\to E$ to \eqref{Formulation:WeakSolution} is strongly continuous and thus there exists a unique generator of the form 
    \begin{align*}
    \mathcal{A}=\mathcal{B}+\mathcal{C}+\mathcal{N}
    \end{align*}
with domain $D(\mathcal{A})\subset E$, see \cite[Thm 1.4, Chapter II]{EngelNagel}. Since $\mathcal{T}(t)\mu_0$ solves the weak formulation \eqref{Formulation:WeakSolution} for an initial measure $\mu_0\in \mathcal{M}(\R^+)$, the map $t\mapsto \mathcal{T}(t)\mu_0$ is differentiable so that we conclude $\mathcal{M}(\R^+)\subset D(\mathcal{A})$. 
\item [ii)]  \label{rem:positivity} Theorem \ref{existenceAgnieszka} implies that $\mathcal{T}(t):\mathcal{M}^+(\R^+)\to \mathcal{M}^+(\R^+)$ for all $t\in [0,T]$, so  $\mathcal{T}$ defines a positive semigroup.
\item [iii)] The operators $\mathcal{C},\mathcal{N}$ are bounded and thus the generated semigroups $\mathcal{T}_{\mathcal{C}}(t)$ and $\mathcal{T}_{\mathcal{N}}(t)$ are strongly continuous, see \cite[Prop. 3.5., Chapter I]{EngelNagel}. Strong continuity of the semigroup $\mathcal{T}_{\mathcal{B}}$ follows from Proposition \ref{stronglycontinuous} with $\mathcal{N}=\mathcal{C}=0$.
\end{itemize}
\end{remark}

\noindent Before we state our main result, we need to recall one concept from semigroup theory.
\begin{definition}\label{irr}
The positive semigroup $\mathcal{T}(t)$ is called \textbf{irreducible}, if for all $0\not\equiv \mu\in \mathcal{M}^+(\R^+),\,0\not\equiv \varphi\in BL(\R^+)_+=E^*_+$ there exists a time $t>0$ such that 
    \begin{align*}
       \langle\mathcal{T}(t)\,\mu,\varphi\rangle_{E,E^*}= \int_{\R^+}\phi\,\ud \mu>0. 
    \end{align*}
Here, the cone is given by
    \begin{align*}
    BL_+(\R^+)=\{\varphi\in BL(\R^+)\mid \varphi\geq 0\}.
    \end{align*}
\end{definition} 

\noindent We are now in the position to formulate our main result, which states that solutions of model \eqref{Model} approach a finite dimensional attractor.  The proof will be given in Section \ref{section:asymptotic_proof}.

\begin{theorem}\label{mainresult}
Under Assumptions \ref{Assum}(i)-(v), one of the following holds true.
\begin{itemize}
\item[(i)] The semigroup $\mathcal{T}(t)$ generated by $\mathcal{B}+\mathcal{C}+\mathcal{N}$ is uniformly exponentially stable, 
i.e. $\omega_0(\mathcal{T})<0$, that is, solutions of model \eqref{Model} tend to zero.
\item[(ii)] There exists a constant $\lambda_*\ge 0$  and a finite rank operator $P_*$ on $E$ such that the semigroup $\mathcal{T}(t)$ decomposes as
    \begin{align*}
    \mathcal{T}(t)=\mathcal{T}_*(t) + Q(t),
    \end{align*}
where
    \begin{align*}
    \mathcal{T}_*(t)=e^{\lambda_*\, t}\sum_{j=0}^{k_*-1}\frac{t^j}{j!}\left(\mathcal{B}+\mathcal{C}+\mathcal{N}-\lambda_*\right)^j\, P_*,
    \end{align*}
and
    \begin{align*}
    ||Q(t)||\le M_\delta e^{(\lambda_*-\delta) t},\quad \text{for some}\quad \delta>0,\, M_\delta\,\ge 1,\quad \forall\, t\ge 0.
    \end{align*}
\item[(iii)]
If in addition the semigroup $\mathcal{T}(t)$ is irreducible, then there exists a rank one operator $P_*$ such that
    \begin{align*}
    \displaystyle\lim_{t\to \infty} e^{-\lambda_*\,t}\,\mathcal{T}(t)=P_*.
    \end{align*}
In other words, the semigroup $\mathcal{T}(t)$ exhibits asynchronous exponential growth and approaches a one-dimensional globally attracting eigenspace associated with the dominant eigenvalue.
\end{itemize}
\end{theorem}

\begin{remark}
We refer to Proposition \ref{irreducible} for an irreducibility condition of the semigroup $\mathcal{T}(t)$.
\end{remark}

\section{Proof of asymptotic behaviour}
\label{section:asymptotic_proof}
\noindent

\noindent The goal of this section is to prove Theorem \ref{mainresult}. Since \eqref{Model} is formulated on an unbounded state space, the associated semigroup is not eventually compact and classical spectral results based on eventual compactness are not directly applicable. We therefore rely on the weaker notion of quasi-compactness, which still provides sufficient spectral information to characterise the asymptotic behaviour of solutions.\\
To analyse the asymptotic behaviour of solutions to \eqref{Model}, it suffices to study the semigroup $\mathcal{T}(t)$ generated by $\mathcal{A}:=\mathcal{B}+\mathcal{C}+\mathcal{N}$. The key quantities for the analysis are the growth bound, the essential growth bound, and the spectral bound of $\mathcal{A}$, which we recall below.\\
The growth bound of $\mathcal{T}(t)$ and the spectral bound of its generator $\mathcal{A}$ are defined by
    \begin{align}
\label{growth_bound}
\omega_0=\omega_0(\mathcal{T}):=\inf\left\{w\in\mathbb{R}\,|\,\exists\, M_w\ge 1,\,\text{such that}\,  ||\mathcal{T}(t)||\le M_w e^{wt},\,\forall\,t\in\mathbb{R}^+ \right\},
    \end{align}
    \begin{align*}
    s(\mathcal{A})=\sup\left\{\text{Re}(\lambda)\,|\,\lambda\in\sigma(\mathcal{A})\right\},
    \end{align*}
where $\sigma(\mathcal{A})$ denotes the spectrum of $\mathcal{A}$. Note that in general
\[
-\infty\le s(\mathcal{A})\le\omega_0<\infty.
\]
For a bounded linear operator $T$ on a Banach space $\mathcal Y$, the essential norm is defined as
    \begin{align*}
	\|T\|_\text{ess}:=\text{dist}\,\left(T, K({\mathcal Y})\right),
    \end{align*}
where $K({\mathcal Y})$ denotes the set of compact linear operators on $\mathcal Y$.\\
The \textbf{essential growth bound} of the semigroup $\mathcal{T}(t)$ is then given by
    \begin{align*}
\omega_\text{ess}\left(\mathcal{T}\right)\left[=\omega_\text{ess}\left(\mathcal{A}\right)\right]
=
\lim_{t\rightarrow \infty}
\left(\frac{\ln \|\mathcal{T}(t)\|_\text{ess}}{t}\right).
    \end{align*}
Moreover, for every compact operator ${\mathcal K}\in K({\mathcal Y})$,
    \begin{align*}
	\omega_\text{ess}\left({\mathcal A}\right)=\omega_\text{ess}\left({\mathcal A}+{\mathcal K}\right).
    \end{align*}
The importance of the essential growth bound stems from the identity
\begin{align}
    \label{alternative_characterisation_w_0}
	\omega_0\left({\mathcal T}\right)= \max\,\left\{\omega_\text{ess}\left({\mathcal T}\right), s\left({\mathcal A}\right)\right\},
\end{align}
see \cite[Cor. 2.11, Ch. IV]{EngelNagel}.
\\
As eventual compactness is not available for the semigroup $\mathcal{T}(t)$, we consider the weaker notion of quasi-compactness, see \cite{EngelNagel}.
\begin{definition}
A strongly continuous semigroup $\mathcal{T}(t)$ on a Banach space $E$ is called \textbf{quasi-compact} if
\begin{equation*}
\liminf_{t\to\infty} \left\{\norma{\mathcal{T}(t)-\mathcal{K}}\,\,\vert\,\, \mathcal{K}\in K(E)\right\}=0.
\end{equation*}
\end{definition}
\noindent Note that, as stated in Proposition 3.5. in \cite[Ch.V]{EngelNagel}, a semigroup $\mathcal{T}(t)$ is quasi-compact if and only if $\omega_{ess}(\mathcal{T})<0$. 
\begin{proposition}\label{quasi-compact}
Under Assumptions \ref{Assum}(i)-(v) the semigroup $\mathcal{T}(t)$ generated by $\mathcal{B}+\mathcal{C}+\mathcal{N}$ is quasi-compact.
\end{proposition}
\begin{proof}
First note that the integral operator $\mathcal{N}$ is compact, since it can be approximated by operators of finite dimensional range. By invoking Proposition 3.6 from \cite[Ch.V]{EngelNagel} it is sufficient to show that the semigroup $\mathcal{T}_{\mathcal{B}+\mathcal{C}}(t)$ generated by $\mathcal{B}+\mathcal{C}$ is quasi-compact. On the grounds of  Proposition 3.5 in \cite[Ch.V]{EngelNagel} it is enough to show that the following inequality holds true 
    \begin{align*}
    \max\left\{s(\mathcal{B}+\mathcal{C}),\omega_{ess}\left(\mathcal{T}_{\mathcal{B}+\mathcal{C}}\right)\right\}=\omega_0\left(\mathcal{T}_{\mathcal{B}+\mathcal{C}}\right)<0;
    \end{align*}
that is, the semigroup $\mathcal{T}_{\mathcal{B}+\mathcal{C}}(t)$ is strictly contractive. To this end, we first compute the adjoint semigroup generated by $(\mathcal{B}+\mathcal{C})^*$, as their operator norms coincide
    \begin{align}
    \label{duality}
    \norma{\mathcal{T}_{\mathcal{B}+\mathcal{C}}}_{\mathcal L\left(( E, \norma{\cdot}_{\BLN^*});(E, \norma{\cdot}_{\BLN^*})\right)} = \norma{\mathcal{T}_{(\mathcal{B}+\mathcal{C})^*}}_{\mathcal L\left( \BL;\BL\right)}.
    \end{align}
To simplify the computations, we will use an equivalent norm on $BL(\R^+)$
    \begin{align}
    \label{new_BL}
    \|\phi\|_{bL}:=\sup_{x\in \R^+}\left[|\phi(x)|+|\partial_x \phi(x)|\right].
\end{align}
and we directly see
    \begin{align}
    \label{equivalence_constants}
    \|\varphi\|_{BL}\leq \|\varphi\|_{bL}\leq 2\|\varphi\|_{BL}.
    \end{align}
To avoid confusion, we will denote the operator norm induced by $\|\cdot\|_{BL}$ with $\|\cdot\|$ and the operator norm induced by $\|\cdot\|_{bL}$ with $\|\cdot\|_{op}$. Using Remark \ref{rem:transport_for_measures}, we see for $\mu\in \mathcal{M}(\R^+)$ and $\varphi\in \BL(\R^+)\cap C^1(\R^+)\subset BL(\R^+)=E^*$ (cf. \ref{remark:measure_space} ii))
    \begin{align*}
    \langle (\mathcal{B}+\mathcal{C})\mu,\varphi\rangle_{E,E^*}=&\int_{\R^+}\varphi(x)\ud[(\mathcal{B}+\mathcal{C})\mu](x)=\int_{\R^+}b(x)\partial_x\varphi(x)-c(x)\varphi(x)\ud\mu(x)\\
    =&\int_{\R^+}(b(x)\partial_x(\cdot)-c(x))\varphi(x)\ud\mu(x)=\langle\mu,(\mathcal{B}+\mathcal{C})^*\varphi\rangle_{E,E^*}, 
    \end{align*}
and we conclude
    \begin{align*}
    (\mathcal{B}+\mathcal{C})^*=b\partial_x-c.
    \end{align*}
The operator $\mathcal{C}^*:BL(\R^+)\to BL(\R^+), \varphi\mapsto -c\varphi$ is bounded which yields $D(\mathcal{C}^*)=BL(\R^+)$. This together with the linearity of the problem implies     \begin{align}\label{linearity_of_problem_adjoint}
    (\mathcal{B}+\mathcal{C})^*=\mathcal{B}^*+\mathcal{C}^*,
    \end{align}
i.e. we can consider two separate problems given by the adjoint operators
    \begin{align*}
    \mathcal{B^*}\,\phi=  b \, \partial_x \phi, \qquad \qquad
    \mathcal{C^*}\,\phi=-c\,\phi 
    \end{align*}
and compute their operator norms separately.

\noindent We start with the semigroup $\mathcal{T}_{\mathcal{B}^*}(t)$ generated by $\mathcal{B^*}$ and claim that it is contractive (see e.g. \cite[Ch.II]{EngelNagel}), i.e. $\left|\left|\mathcal{T}_{\mathcal{B}^*}(t)\right|\right|_{op}\le 1,\,\forall\,t\ge 0$. To see this, note that the induced Cauchy problem 
    \begin{align*}
    \partial_t\varphi(t,x)=\mathcal{B}^*\varphi(t,x), \qquad \varphi(0,x)=\varphi_0(x)
    \end{align*}
is solved by 
    \begin{align}
    \label{T_B*}
    \mathcal{T}_{\mathcal{B}^*}(t)\varphi_0=\varphi_0(X_b(t,x)),
    \end{align}
where $X_b$ denotes the flow of $b$ which is defined in \eqref{flow}.
Next we need to compute $\partial_x X_b(t,x)$ as this expression will appear later. Note that  
    \begin{align*}
    \partial_t\left(\partial_x X_b(t,x)\right)=\partial_x\partial_t X_b(t,x)=\partial_xb(X_b(t,x))
    =b'(X_b(t,x))\partial_x X_b(t,x),
    \end{align*}
which implies 
    \begin{align}
    \label{derivative}
\partial_x X_b(t,x)=C\exp\left(\int_0^tb'(X_b(s,x))\,\ud s\right).
    \end{align}
As 
    \begin{align*}
    C=\partial_x X_b(t,x)\mid_{t=0}=\partial_x X_b(0,x)=\partial_x I(x)=1, 
    \end{align*}
we conclude $C=1$.\\
Now we can compute the operator norm of $\mathcal{T}_{\mathcal{B}^*}$. Let $t\geq 0$ and $\varphi\in BL(\R^+)$ with $\|\varphi\|_{bL}\leq 1$. Then using \eqref{T_B*},  \eqref{derivative} as well as $b'\leq 0$  we see
    \begin{align*}
\|\mathcal{T}_{B^*}(t)\varphi\|_{bL^*}
=&\sup_{x\in \R^+}\left[|\varphi(X_b(t,x))|+|\partial_x\varphi(X_b(t,x)|\right]\\
=&\sup_{x\in \R^+}\left[|\varphi(X_b(t,x))|+|\varphi'(X_b(t,x))\partial_x X_b(t,x)|\right]\\
=&\sup_{x\in \R^+}\left[|\varphi(X_b(t,x))|+|\varphi'(X_b(t,x))\exp\left(\int_0^tb'(X_b(s,x))\ud s\right)|\right]\\
\leq& \sup_{x\in \R^+}\left[|\varphi(X_b(t,x))|+|\varphi'(X_b(t,x))|\right]=\|\varphi\|_{bL}, 
    \end{align*}
so that indeed $\|\mathcal{T}_{\mathcal{B}*}(t)\|_{op}\leq 1$ for all $t\geq 0$.\\
From the definition of the growth bound $\omega_0$ (see \eqref{growth_bound}) we conclude      
    \begin{align*}
         \omega_0\left(\mathcal{T}_{\mathcal{B}^*}\right)\le 0.
    \end{align*}
In a second step, we show that the bounded operator $\mathcal{C}^*$ generates a positive contraction semigroup $\mathcal T_{\mathcal{C}^*}(t)$ satisfying $\norma{\mathcal T_{\mathcal{C}^*}(t)}_{op}\le e^{-\kappa t},\, t\ge 0$, for some $\kappa>0$. 
The induced Cauchy problem
    \begin{align*}
        \partial_t\varphi(t,x)=\mathcal{C}^*\varphi(t,x), \qquad \varphi(0,x)=\varphi_0(x)
    \end{align*}
is solved by
    \begin{align}
    \label{T_C*}
    \mathcal{T}_{\mathcal{C}^*}(t)\varphi_0(x)=
    \varphi_0(x)e^{-c(x)t},
    \end{align}
so that we compute for $t\geq 0$ and $\varphi$ with $\|\varphi\|_{bL}\leq 1$
    \begin{align}
        \begin{split}
            \label{T_C*}
\|\mathcal{T}_{\mathcal{C}^*}(t)\varphi\|_{bL}=&\sup_{x\in \R^+}\left [|\varphi(x)e^{-c(x)t}|+|\partial_x(\varphi(x)e^{-c(x)t})|\right]\\
=&\sup_{x\in \R^+}\left[|\varphi(x)|e^{-c(x)t}+|\varphi'(x)e^{-c(x)t}-\varphi(x)c'(x)te^{-c(x)t}|\right]\\
\leq &\sup_{x\in \R^+}\left[|\varphi(x)|e^{-c(x)t}(1+t|c'(x)|)+|\varphi'(x)|e^{-c(x)t}\right]\\
\leq& \sup_{x\in \R^+}\left[|\varphi(x)|e^{(-c(x)+|c'(x)|)t}+|\varphi'(x)|e^{-c(x)t}\right]\\
\leq& e^{-\kappa t}\sup_{x\in \R^+}\left[|\varphi(x)|+|\varphi'(x)|\right]=e^{-\kappa t}\|\varphi\|_{bL}.
        \end{split}
\end{align}
Assumption \ref{Assum} v) was used in the last inequality. From \eqref{T_C*} we conclude for all $t\geq 0$ $\|\mathcal{T}_{C^*}(t)\|_{op}\leq e^{-\kappa t}$ and in particular we have for all $n\in \mathbb{N}$ and $t\geq 0$
    \begin{align*}
        \left\|\left[\mathcal{T}_{\mathcal{B}^*}\left(\frac t n\right)\mathcal{T}_{\mathcal{C}^*}\left(\frac t n\right)\right]^n\right\|_{op}
        \leq\left\|\mathcal{T}_{\mathcal{B}^*}\left(\frac t n\right)\right\|^n_{op}\left\|\mathcal{T}_{\mathcal{C}^*}\left(\frac t n\right)\right\|^n_{op}\leq 1^n \left(e^{-\kappa \frac t n}\right)^n=e^{-\kappa t}.
    \end{align*}
From \eqref{T_B*} and \eqref{T_C*} we directly see that the semigroups $\mathcal{T}_{\mathcal{B}^*}$ and $\mathcal{T}_{\mathcal{C}^*}$ are strongly continuous. An application of the Trotter product formula (see e.g. Corollary 5.8  in \cite[Ch.III]{EngelNagel}) yields that the semigroup $\mathcal{T}_{\mathcal{B}^*+\mathcal{C}^*}(t)$ generated by $\mathcal{B}^*+\mathcal{C}^*$ satisfies 
    \begin{align*}
    \left\|\mathcal{T}_{\mathcal{B}^*+\mathcal{C}^*}(t)\right\|_{op}\le \exp\left(-\kappa\,t\right),\quad \forall\,t\ge 0.
    \end{align*}
This together with the duality \eqref{duality}, observation \eqref{linearity_of_problem_adjoint} and the equivalence \eqref{equivalence_constants} implies 
    \begin{align*}
    \left\|\mathcal{T}_{\mathcal{B}+\mathcal{C}}(t)\right\|=\left\|\mathcal{T}_{(\mathcal{B}+\mathcal{C})^*}(t)\right\|
    =\left\|\mathcal{T}_{\mathcal{B}^*+\mathcal{C}^*}(t)\right\|\leq \left\|\mathcal{T}_{\mathcal{B}^*+\mathcal{C}^*}(t)\right\|_{op}
    \le \exp\left(-\kappa\,t\right),\quad \forall\,t\ge 0.
    \end{align*}
In particular,  $\omega_0\left(\mathcal{T}_{\mathcal{B}+\mathcal{C}}\right)\le -\kappa<0$, completing the proof.
\end{proof}

\begin{remark}
\label{rem:kappa}
The proof of Proposition \ref{quasi-compact} in fact shows that $\omega_{ess}(\mathcal{T})\le -\kappa$, i.e. the essential spectrum of $\mathcal{T}(t)$ is contained in the left half-plane $\left\{\lambda\in\mathbb{C}\,|\,\text{Re}(\lambda)\le -\kappa\right\}$.
\end{remark}

\noindent The significance of quasi-compactness of a semigroup $\mathcal{T}(t)$ is demonstrated by the following cha\-rac\-te\-ri\-sation theorem, recalled from \cite{EngelNagel} for the reader's convenience.
\begin{theorem}\cite[Ch.V, Theorem 3.7]{EngelNagel}\label{quasi-compact-char}
Let $\mathcal{T}(t)$ be a quasi-compact and strongly con\-ti\-nuous semigroup with generator $\mathcal{A}$ on the Banach space $E$. Then, the following holds.
\begin{itemize}
\item[(i)] The set $\left\{\lambda\in\sigma(\mathcal{A})\,|\, Re(\lambda)\ge 0\right\}$ is finite (possibly empty!), and consists of poles of the resolvent operator $R(\cdot, \mathcal{A})$ of finite algebraic multiplicity.
\item[(ii)] If we denote the set of poles by $\lambda_1,\cdots,\lambda_m$ and their corresponding residues by $P_1,\cdots,P_m$, with orders $k_1,\cdots,k_m$, respectively, then we have
\begin{align*}
\mathcal{T}(t)=\mathcal{T}_1(t)+\cdots+\mathcal{T}_m(t)+\mathcal{R}(t),
\end{align*}
where
\begin{align*}
\mathcal{T}_n(t)=e^{\lambda_n\,t}\displaystyle\sum_{j=0}^{k_n-1}\frac{t^j}{j!}(\mathcal{A}-\lambda_n)^j\,P_n,\quad t\ge 0,\quad 1\le n\le m,
\end{align*}
and
\begin{align*}
||\mathcal{R}(t)||\le M\,e^{-\varepsilon\,t},\quad \text{for some} \quad \varepsilon>0,\,M\ge 1,\quad \forall\,t\ge 0.
\end{align*}
\end{itemize}
\end{theorem} 
\noindent We are now in the position to prove our main Theorem \ref{mainresult}.

\begin{proofof}{Theorem \ref{mainresult}}
As a first step, we apply Theorem \ref{quasi-compact-char} to the strongly continuous and quasi-compact semigroup $\mathcal{T}(t)$ generated by $\mathcal{A}=\mathcal{B}+\mathcal{C}+\mathcal{N}$. By (i), 
    \begin{align}
    \label{set_of_poles}
      \{\lambda \in \sigma(\mathcal{A})\mid\mathrm{Re}(\lambda)\geq 0\}=\{\lambda_1,...,\lambda_m\}. 
    \end{align}
If the spectrum of $\mathcal{A}$ is empty, then by definition $s(\mathcal{A})=-\infty$ and we conclude with Remark \ref{rem:kappa} and \eqref{alternative_characterisation_w_0}	
    \begin{align*}
    \omega_0\left({\mathcal T}\right)= \max\,\left\{\omega_\text{ess}\left({\mathcal T}\right), s\left({\mathcal A}\right)\right\}\leq-\kappa<0,
    \end{align*}
i.e. $\mathcal{T}(t)$ is uniformly exponentially stable.\\
So for the rest of the proof we assume $\sigma(\mathcal{A})\neq \emptyset$. According to Remark \ref{rem:positivity} ii) the semigroup $\mathcal{T}(t)$ is positive which implies $s(\mathcal{A})\in\sigma(\mathcal{A})$, see \cite[Ch. VI, Theorem 1.10]{EngelNagel}. If  $s(\mathcal{A})< 0$, then the set in \eqref{set_of_poles} is empty and we conclude $\omega_0(\mathcal{T})<0$ as before. However, if $s(\mathcal{A})\geq 0$, then the finite set of poles  \eqref{set_of_poles} is not empty and consists of eigenvalues of $\mathcal{T}$, see \cite[VIII.8, Theorem 3]{Yosida}. Set 
    \begin{align*}
    \lambda_*:=s(\mathcal{A})=\max\{\mathrm{Re}(\lambda_1),...,\mathrm{Re}(\lambda_m)\}.
    \end{align*}
Then $\lambda_*\geq 0$ is an isolated eigenvalue of finite algebraic multiplicity. Let $P_*$ be the corresponding spectral projection which has finite rank, see \cite[Ch. IV, Cor 2.11]{EngelNagel}. According to Theorem 3.1 in \cite[Ch.V]{EngelNagel} the semigroup $\mathcal{T}(t)$ decomposes as 
    \begin{align*}
    \mathcal{T}(t)=\mathcal{T}_*(t) + Q(t),
    \end{align*}
where
    \begin{align*}
    \mathcal{T}_*(t)=e^{\lambda_*\, t}\sum_{j=0}^{k_*-1}\frac{t^j}{j!}\left(\mathcal{A}-\lambda_*\right)^j\, P_*,
    \end{align*}
Furthermore, for any $\delta>0$ with
    \begin{align*}
    \lambda_*-\delta>\sup\{\omega_{ess}(\mathcal{T})\}\cup\{\mathrm{Re}(\lambda)\mid \lambda\in \sigma(\mathcal{A})\setminus\{\lambda_*\}\}, 
    \end{align*}
there exists $M_{\delta}>0$ such that we have an estimate of the form
    \begin{align*}
        ||Q(t)||\le M_\delta e^{(\lambda_*-\delta) t} \qquad \forall t\geq 0,
    \end{align*}
which shows (ii). If the semigroup $\mathcal{T}(t)$ additionally is irreducible, then $\lambda_*=s(\mathcal{A})$ is a dominant eigenvalue of multiplicity one with residue $P_*$ of rank one, see  \cite[C-III, Proposition 3.5]{Arendt}.
This implies that the semigroup $\mathcal{T}(t)$ exhibits asynchronous exponential growth, see e.g. \cite{Webb1987}, i.e.
\begin{align*}
    \left|\left|e^{-\lambda_*\,t}\,\mathcal{T}(t)-P_*\right|\right|\le M\,e^{-\varepsilon\,t},
    \end{align*}
for some $M\ge 1$ and $\varepsilon>0$. This completes the proof.
\end{proofof}

\noindent Next we formulate a sufficient condition for the semigroup $\mathcal{T}(t)$ generated by $\mathcal{B}+\mathcal{C}+\mathcal{N}$ to be irreducible.

\begin{proposition}\label{irreducible}
If there exists a $\hat{y}>0$, such that $0\in\displaystyle\bigcap_{y>\hat{y}}\text{supp}\,(\eta(y))$, then the semigroup $\mathcal{T}(t)$ generated by $\mathcal{B}+\mathcal{C}+\mathcal{N}$ is irreducible. 
\end{proposition}
\begin{proof}
First note that the characterisation of irreducibility given in Definition \ref{irr}  is equivalent to the following condition: 
    \begin{align}
    \label{equivalent_irreducibility_condition}
    \forall\,\mu\in \mathcal{M}^+(\R^+),\,\mu\not\equiv 0\quad \text{we have}\quad \displaystyle\bigcup_{t\ge 0} \text{supp}(\mathcal{T}(t)\,\mu)=\mathbb{R}^+.
    \end{align}
Indeed, for the first implication assume by way of contradiction that  there exists $0\neq\mu\in \mathcal{M}^+(\R^+)$ with
    \begin{align*}
        \bigcup_{t\geq 0}\text{supp}(\mathcal{T}(t)\mu)\neq \R^+.
    \end{align*}
As the support is closed, there exists $x\in \R^+$ and $r>0$ with $\mathcal{T}(t)\mu(x-r,x+r)=0$. Now consider a nonnegative function $\varphi\in BL(\R^+)=E^*$, see Remark \ref{remark:measure_space} ii), supported in $(x- r/2,x+r/2)$. Then 
    \begin{align*}
    \langle \mathcal{T}(t)\mu,\varphi\rangle_{E,E^*}=\int_{\R^+}\varphi(y)\ud[\mathcal{T}(t)\mu](y)=\int_{\left(x-\frac r 2,x+\frac r 2\right)}\varphi(y)\ud[\mathcal{T}(t)\mu](y)=0,
    \end{align*}
which contradicts the definition of irreducibility in Definition \ref{irr}. For the other direction, assume there exists $0\neq\mu\in \mathcal{M}^+(\R^+)$ and $0\neq \varphi\in BL_+(\R^+)$ such that for all $t>0$
    \begin{align*}
    \langle \mathcal{T}(t)\mu,\varphi\rangle_{E,E^*}=0.
    \end{align*}
As $\varphi\neq 0$, there exists an interval $I\subseteq \R^+$ with $\varphi\mid_I>0$. Using \eqref{equivalent_irreducibility_condition} and restricting $I$ if necessary we can assume the existence of a time $t>0$  with $\mathcal{T}(t)\mu(I)>0$. Hence,
    \begin{align*}
        0=\int_{\R^+}\varphi(y)\ud[\mathcal{T}(t)\mu](y)\geq  \int_I \varphi(y)\ud[\mathcal{T}(t)\mu](y)>0.
    \end{align*}
In particular, \eqref{equivalent_irreducibility_condition} is indeed an equivalent condition for irreducibility.\\
Next we note that if for any measure $\mu\in \mathcal{M}^+(\R^+)$ there is a time $t^*\ge 0$ such the point $y^*\in \text{supp}(\mathcal{T}(t^*)\,\mu)$,
then for all $y>y^*$ there exists $t\ge t^*$ with $y\in\text{supp}(\mathcal{T}(t)\,\mu)$. This is due to the strict positivity of $b$ and the boundedness of $c$. \\
\noindent  Consequently, it is sufficient to check that $0\in \mathrm{supp}(\mathcal{T}(t)\mu)$ for some $t$. To this end, take any $0\not\equiv \mu\in \mathcal{M}^+(\R^+)$. By the previous observation (and since $\mu$ is not the zero element), there exist $t^*\ge 0$ and $y^*>\hat{y}$ with  $y^*\in\mathrm{supp}(\mathcal{T}(t^*)\,\mu)$. As $0\in \mathrm{supp}\,\eta(y^*)$, it follows that $0\in \mathrm{supp}(\mathcal{T}(t)\,\mu)$ for every $t>t^*$, which concludes the proof.
\end{proof}

\noindent Note that from the biological point of view the irreducibility condition above is very natural, it requires (when interpreting the structuring variable $x$ as individual size for example) that large individuals produce offspring of minimal size, see e.g. \cite{CDF}.

\section{Acknowledgements}
\noindent CD and AMC were supported by the European Research Council (ERC) under the European Union’s Horizon 2020 research and innovation programme (synergy project PEPS, no. 101071786). PG was supported by  NCN UMO-2023/51/B/ST1/01546.  AMC gratefully acknowledges support of 4EU+ Alliance Visiting Professorships Programme for a research stay at the University of Warsaw.

\bibliographystyle{abbrv}
\bibliography{bibliography}

@article{AcklehFarkas,
  author  = {A. S. Ackleh and J. Z. Farkas},
  title   = {On the net reproduction rate of continuous structured populations with distributed states at birth},
  journal = {Computers \& Mathematics with Applications},
  volume  = {66},
  pages   = {1685--1694},
  year    = {2013}
}

@article{Ackleh1,
  author  = {A. S. Ackleh and K. Ito},
  title   = {Measure-valued solutions for a hierarchically size-structured population},
  journal = {Journal of Differential Equations},
  volume  = {217},
  pages   = {431--455},
  year    = {2005}
}

@book{Amman,
  author    = {H. Amman and J. Escher},
  title     = {Analysis II},
  publisher = {Birkh{\"a}user},
  address   = {Basel},
  year      = {2006}
}

@book{Arendt,
  author    = {W. Arendt and A. Grabosch and G. Greiner and U. Groh and H. P. Lotz and U. Moustakas and R. Nagel and F. Neubrander and U. Schlotterbeck},
  title     = {One-Parameter Semigroups of Positive Operators},
  publisher = {Springer},
  address   = {Berlin},
  year      = {1986}
}

@book{MagalRuan,
  author    = {P. Auger and P. Magal and S. Ruan},
  title     = {Structured Population Models in Biology and Epidemiology},
  series    = {Lecture Notes in Mathematics},
  volume    = {1936},
  publisher = {Springer},
  address   = {Berlin},
  year      = {2008}
}

@article{EBT,
  author  = {{\AA}. Br{\"a}nnstr{\"o}m and L. Carlsson and D. Simpson},
  title   = {On the convergence of the escalator boxcar train},
  journal = {SIAM Journal on Numerical Analysis},
  volume  = {51},
  pages   = {3213--3231},
  year    = {2013}
}

@article{BCMC,
  author  = {J. E. Busse and S. Cuadrado and A. Marciniak-Czochra},
  title   = {Local asymptotic stability of a system of integro-differential equations describing clonal evolution of a self-renewing cell population under mutation},
  journal = {Journal of Mathematical Biology},
  volume  = {84},
  pages   = {1--36},
  year    = {2022}
}

@article{BGMC,
  author  = {J.-E. Busse and P. Gwiazda and A. Marciniak-Czochra},
  title   = {Mass concentration in a nonlocal model of clonal selection},
  journal = {Journal of Mathematical Biology},
  volume  = {73},
  pages   = {1001--1033},
  year    = {2016}
}

@article{BuergerBomze,
  author  = {R. B{\"u}rger and I. M. Bomze},
  title   = {Stationary distributions under mutation-selection balance: structure and properties},
  journal = {Advances in Applied Probability},
  volume  = {28},
  pages   = {227--251},
  year    = {1996}
}

@article{CalsinaCuadrado,
  author  = {{\`A}. Calsina and S. Cuadrado},
  title   = {Small mutation rate and evolutionarily stable strategies in infinite dimensional adaptive dynamics},
  journal = {Journal of Mathematical Biology},
  volume  = {48},
  pages   = {135--159},
  year    = {2004}
}

@article{CalsinaCuadrado2,
  author  = {{\`A}. Calsina and S. Cuadrado},
  title   = {Stationary solutions of a selection mutation model: The pure mutation case},
  journal = {Mathematical Models and Methods in Applied Sciences},
  volume  = {15},
  pages   = {1091--1117},
  year    = {2005}
}

@article{CDF,
  author  = {{\`A}. Calsina and O. Diekmann and J. Z. Farkas},
  title   = {Structured populations with distributed recruitment: from PDE to delay formulation},
  journal = {Mathematical Methods in the Applied Sciences},
  volume  = {39},
  pages   = {5175--5191},
  year    = {2016}
}

@article{CCGU,
  author  = {J. A. Carrillo and R. M. Colombo and P. Gwiazda and A. Ulikowska},
  title   = {Structured populations, cell growth and measure valued balance laws},
  journal = {Journal of Differential Equations},
  volume  = {252},
  pages   = {3245--3277},
  year    = {2012}
}

@article{CGU,
  author  = {J. A. Carrillo and P. Gwiazda and A. Ulikowska},
  title   = {Splitting-Particle Methods for Structured Population Models: Convergence and Applications},
  journal = {Mathematical Models and Methods in Applied Sciences},
  volume  = {24},
  pages   = {2171--2197},
  year    = {2014}
}

@book{Clement,
  author    = {Ph. Cl{\'e}ment and H. J. A. M. Heijmans and S. Angenent and C. J. van Duijn and B. de Pagter},
  title     = {One-Parameter Semigroups},
  publisher = {North-Holland},
  address   = {Amsterdam},
  year      = {1987}
}

@article{Desvillettes,
  author  = {L. Desvillettes and P. E. Jabin and S. Mischler and G. Raoul},
  title   = {On mutation-selection dynamics},
  journal = {Communications in Mathematical Sciences},
  volume  = {6},
  pages   = {729--747},
  year    = {2008}
}

@book{Duell_book,
  author    = {C. D{\"u}ll and P. Gwiazda and A. Marciniak-Czochra and J. Skrzeczkowski},
  title     = {Spaces of Measures and Their Applications to Structured Population Models},
  publisher = {Cambridge University Press},
  year      = {2021}
}

@phdthesis{Duell_thesis,
  title        = {Generalising nonlinear population models- Radon measures, Polish spaces and the flat norm},
  author       = {Düll, Christian},
  year         = 2024,
  note         = {Available at https://archiv.ub.uni-heidelberg.de/volltextserver/view/creators/index.D.html},
  school       = {Heidelberg University},
  type         = {PhD thesis}
}

@article{M3AS_measures,
  author  = {C. D{\"u}ll and P. Gwiazda and A. Marciniak-Czochra and J. Skrzeczkowski},
  title   = {Structured Population Models on Polish Spaces: A Unified Approach Including Graphs, Riemannian Manifolds and Measure Spaces to Describe Dynamics of Heterogeneous Populations},
  journal = {Mathematical Models and Methods in Applied Sciences},
  volume  = {34},
  number  = {1},
  pages   = {109--143},
  year    = {2024}
}

@book{EngelNagel,
  author    = {K.-J. Engel and R. Nagel},
  title     = {One-Parameter Semigroups for Linear Evolution Equations},
  publisher = {Springer-Verlag},
  year      = {2000}
}

@article{EversHilleMuntean,
  author  = {J. Evers and S. Hille and A. Muntean},
  title   = {Mild solutions to a measure-valued mass evolution problem with flux boundary conditions},
  journal = {Journal of Differential Equations},
  volume  = {259},
  pages   = {1068--1097},
  year    = {2015}
}

@article{FarkasHagen,
  author  = {J. Z. Farkas and T. Hagen},
  title   = {Asymptotic analysis of a size-structured cannibalism model with infinite dimensional environmental feedback},
  journal = {Communications on Pure and Applied Analysis},
  volume  = {8},
  pages   = {1825--1839},
  year    = {2009}
}

@article{FarkasHinow,
  author  = {J. Z. Farkas and D. M. Green and P. Hinow},
  title   = {Semigroup analysis of structured parasite populations},
  journal = {Mathematical Modelling of Natural Phenomena},
  volume  = {5},
  pages   = {94--114},
  year    = {2010}
}

@article{FournierPerthame1,
  author  = {N. Fournier and B. Perthame},
  title   = {A transport distances for PDEs: The coupling method},
  journal = {EMS Surveys in Mathematical Sciences},
  volume  = {7},
  pages   = {1--31},
  year    = {2021}
}

@article{FournierPerthame2,
  author  = {N. Fournier and B. Perthame},
  title   = {A non-expanding transport distance for some structured equations},
  journal = {SIAM Journal on Mathematical Analysis},
  volume  = {53},
  number  = {6},
  pages   = {6847--6872},
  year    = {2021}
}

@article{GLMC,
  author  = {P. Gwiazda and T. Lorenz and A. Marciniak-Czochra},
  title   = {A nonlinear structured population model: Lipschitz continuity of measure valued solutions with respect to model ingredients},
  journal = {Journal of Differential Equations},
  volume  = {248},
  pages   = {2703--2735},
  year    = {2010}
}

@article{GMC,
  author  = {P. Gwiazda and A. Marciniak-Czochra},
  title   = {Structured population equations in metric spaces},
  journal = {Journal of Hyperbolic Differential Equations},
  volume  = {7},
  pages   = {733--773},
  year    = {2010}
}

@article{GJMCU,
  author  = {P. Gwiazda and J. Jab{\l}o\'nski and A. Marciniak-Czochra and A. Ulikowska},
  title   = {Analysis of particle methods for structured population models with nonlocal boundary term in the framework of bounded {L}ipschitz distance},
  journal = {Numerical Methods for Partial Differential Equations},
  volume  = {30},
  pages   = {1797--1820},
  year    = {2014}
}

@article{MagalWebb,
  author  = {P. Magal and G. F. Webb},
  title   = {Mutation, selection, and recombination in a model of phenotype evolution},
  journal = {Discrete and Continuous Dynamical Systems},
  volume  = {6},
  pages   = {221--236},
  year    = {2000}
}

@book{MetzDiekmann,
  editor    = {J. A. J. Metz and O. Diekmann},
  title     = {The Dynamics of Physiologically Structured Populations},
  series    = {Lecture Notes in Mathematics},
  volume    = {68},
  publisher = {Springer-Verlag},
  address   = {Berlin},
  year      = {1986}
}

@article{MischlerScher,
  author  = {S. Mischler and P. Scher},
  title   = {Spectral analysis of semigroups and growth-fragmentation equations},
  journal = {Annales de l'Institut Henri Poincaré - Analyse Non Linéaire},
  volume  = {33},
  pages   = {849--898},
  year    = {2015}
}

@book{Perthame,
  author    = {B. Perthame},
  title     = {Transport Equations in Biology},
  series    = {Frontiers in Mathematics},
  publisher = {Birkhäuser},
  address   = {Basel},
  year      = {2007}
}

@article{Piccoli,
  author  = {B. Piccoli and F. Rossi},
  title   = {Generalized {W}asserstein distance and its application to transport equations with source},
  journal = {Archive for Rational Mechanics and Analysis},
  volume  = {211},
  pages   = {335--358},
  year    = {2014}
}

@book{Sch,
  author    = {H. H. Sch\"{a}fer},
  title     = {Banach lattices and positive operators},
  publisher = {Springer-Verlag},
  address   = {Berlin},
  year      = {1974}
}

@article{Thieme1998DCDS,
  author  = {H. R. Thieme},
  title   = {Positive perturbation of operator semigroups: growth bounds, essential compactness, and asynchronous exponential growth},
  journal = {Discrete and Continuous Dynamical Systems},
  volume  = {4},
  pages   = {735--764},
  year    = {1998}
}

@article{Thieme1998JMAA,
  author  = {H. R. Thieme},
  title   = {Balanced exponential growth of operator semigroups},
  journal = {Journal of Mathematical Analysis and Applications},
  volume  = {223},
  pages   = {30--49},
  year    = {1998}
}

@book{Thieme,
  author    = {H. R. Thieme},
  title     = {Mathematics in Population Biology},
  series    = {Princeton Series in Theoretical and Computational Biology},
  publisher = {Princeton University Press},
  address   = {Princeton, NJ},
  year      = {2003}
}

@article{two_sex,
  author  = {A. Ulikowska},
  title   = {An age-structured, two-sex model in the space of Radon measures: well posedness},
  journal = {Kinetic and Related Models},
  volume  = {5},
  pages   = {873--900},
  year    = {2012}
}

@article{WebbGrabosch,
  author  = {G. F. Webb and A. Grabosch},
  title   = {Asynchronous exponential growth in transition probability models of the cell cycle},
  journal = {SIAM Journal on Mathematical Analysis},
  volume  = {18},
  pages   = {897--908},
  year    = {1987}
}

@article{Webb1987,
  author  = {G. F. Webb},
  title   = {An operator-theoretic formulation of asynchronous exponential growth},
  journal = {Transactions of the American Mathematical Society},
  volume  = {303},
  pages   = {751--763},
  year    = {1987}
}

@book{Webb,
  author    = {G. F. Webb},
  title     = {Nonlinear Age-Dependent Population Dynamics},
  publisher = {Dekker},
  year      = {1985}
}

@article{vanGaans,
title = "Subspaces of normed Riesz spaces",
author = "{van Gaans}, OW",
year = "2004",
volume = "8",
pages = "143--164",
journal = "Positivity: an international journal devoted to the theory and applications of positivity in analysis",
issn = "1385-1292",
publisher = "Springer",
number = "2",
}

@book {Yosida,
    AUTHOR = {Yosida, K\B{o}saku},
     TITLE = {Functional analysis},
    SERIES = {Classics in Mathematics},
      NOTE = {Reprint of the sixth (1980) edition},
 PUBLISHER = {Springer-Verlag, Berlin},
      YEAR = {1995},
     PAGES = {xii+501},
      ISBN = {3-540-58654-7      },
   MRCLASS = {46-01 (47-01)},
  MRNUMBER = {1336382},
       DOI = {10.1007/978-3-642-61859-8     },
       URL = {https://doi.org/10.1007/978-3-642-61859-8  }, 
}

@book {EvansGariepy,
    AUTHOR = {Evans, Lawrence C. and Gariepy, Ronald F.},
     TITLE = {Measure theory and fine properties of functions},
    SERIES = {Textbooks in Mathematics},
   EDITION = {Revised},
 PUBLISHER = {CRC Press, Boca Raton, FL},
      YEAR = {2015},
     PAGES = {xiv+299},
      ISBN = {978-1-4822-4238-6},
   MRCLASS = {28-01},
  MRNUMBER = {3409135},
}

@article {MR4330732,
    AUTHOR = {Ackleh, Azmy S. and Lyons, Rainey and Saintier, Nicolas},
     TITLE = {A structured coagulation-fragmentation equation in the space
              of {R}adon measures: unifying discrete and continuous models},
   JOURNAL = {ESAIM Math. Model. Numer. Anal.},
  FJOURNAL = {ESAIM. Mathematical Modelling and Numerical Analysis},
    VOLUME = {55},
      YEAR = {2021},
    NUMBER = {5},
     PAGES = {2473--2501},
      ISSN = {2822-7840,2804-7214},
   MRCLASS = {35L60 (35Q92 45K05 65M06 92D25)},
  MRNUMBER = {4330732},
       DOI = {10.1051/m2an/2021061},
       URL = {https://doi.org/10.1051/m2an/2021061},
}

@book{Webb85,
year = {1985},
author = {Webb, Glenn F.},
address = {New York [u.a},
booktitle = {Theory of nonlinear age dependent population dynamics},
isbn = {0824772903},
language = {eng},
publisher = {Dekker},
series = {Pure and applied mathematics 89},
title = {Theory of nonlinear age dependent population dynamics },

}

@article{KangShigui,
author = {Kang, Hao and Ruan, Shigui},
year = {2021},
month = {10},
pages = {1-49},
title = {Principal spectral theory and asynchronous exponential growth for age-structured models with nonlocal diffusion of Neumann type},
volume = {384},
journal = {Mathematische Annalen},
doi = {10.1007/s00208-021-02270-y}
}

@article {MR2287895,
    AUTHOR = {Lasota, Andrzej and Szarek, Tomasz},
     TITLE = {Lower bound technique in the theory of a stochastic
              differential equation},
   JOURNAL = {J. Differential Equations},
  FJOURNAL = {Journal of Differential Equations},
    VOLUME = {231},
      YEAR = {2006},
    NUMBER = {2},
     PAGES = {513--533},
      ISSN = {0022-0396},
   MRCLASS = {60H15 (35R60 37A30 47D07 60J35)},
  MRNUMBER = {2287895},
MRREVIEWER = {Dieter H. Mayer},
       DOI = {10.1016/j.    jde.2006.04.018},
       URL = {https://doi.org/10.1016/j          .jde.2006.04.018},  
}

@article {MR2271485,
    AUTHOR = {Szarek, Tomasz},
     TITLE = {Feller processes on nonlocally compact spaces},
   JOURNAL = {Ann. Probab.},
  FJOURNAL = {The Annals of Probability},
    VOLUME = {34},
      YEAR = {2006},
    NUMBER = {5},
     PAGES = {1849--1863},
      ISSN = {0091-1798},
   MRCLASS = {60J05 (37A30)},
  MRNUMBER = {2271485},
MRREVIEWER = {Niels Jacob},
       DOI = {10.1214/009117906000000313       },
       URL = {https://doi.org/10.1214/009117906000000313           },   
}

@article {MR2995659,
    AUTHOR = {Szarek, Tomasz and Worm, Dani\"{e}l T. H.},
     TITLE = {Ergodic measures of {M}arkov semigroups with the e-property},
   JOURNAL = {Ergodic Theory Dynam. Systems},
  FJOURNAL = {Ergodic Theory and Dynamical Systems},
    VOLUME = {32},
      YEAR = {2012},
    NUMBER = {3},
     PAGES = {1117--1135},
      ISSN = {0143-3857},
   MRCLASS = {60J25 (37A30 47A35 47D06 47D07)},
  MRNUMBER = {2995659},
MRREVIEWER = {Wojciech Bartoszek},
       DOI = {10.1017/S0143385711000022          },
       URL = {https://doi.org/10.1017/S0143385711000022                  },   
}

@article{Thieme_2022,
title = {Discrete-time dynamics of structured populations via {F}eller kernels},
journal = {Discrete and Continuous Dynamical Systems - B},
volume = {27},number = {2},pages = {1091-1119},
year = {2022},
issn = {1531-3492},
doi = {10.3934/dcdsb.2021082   },
url = {/article/id/98b37710-f15a-4d56-afe2-afd87cb8c633},
author = {Horst R. Thieme},
}

@incollection {MR2857021,
    AUTHOR = {Hairer, Martin and Mattingly, Jonathan C.},
     TITLE = {Yet another look at {H}arris' ergodic theorem for {M}arkov
              chains},
 BOOKTITLE = {Seminar on {S}tochastic {A}nalysis, {R}andom {F}ields and
              {A}pplications {VI}},
    SERIES = {Progr. Probab.},
    VOLUME = {63},
     PAGES = {109--117},
 PUBLISHER = {Birkh\"{a}user/Springer Basel AG, Basel},
      YEAR = {2011},
      ISBN = {978-3-0348-0020-4},
   MRCLASS = {60J05 (37A30 37A50 47D07)},
  MRNUMBER = {2857021},
MRREVIEWER = {Wojciech\ Bartoszek},
       DOI = {10.1007/978-3-0348-0021-1\_7},
       URL = {https://doi.org/10.1007/978-3-0348-0021-1_7}                ,
}

\end{document}